\documentclass[12pt,a4paper]{article}

\usepackage{amsmath,amssymb}
\usepackage{amsthm}

\newtheorem{lemma}{Lemma}
\newtheorem{proposition}{Proposition}
\newtheorem{remark}{Remark}
\newtheorem{theorem}{Theorem}

\newtheorem*{corollary}{Corollary}

\title{Discrete diagonal recurrences and\\ discrete minimal submanifolds}

\author{Cristian Ghiu$^1$, Raluca Tulig\u{a}$^2$, \\Constantin Udri\c ste$^2$, Ionel \c Tevy$^2$}
\date{}

\begin{document}

\maketitle

\begin{center}
{\footnotesize
$^1$University Politehnica of Bucharest,
Faculty of Applied Sciences,

Department of Mathematical Methods and Models,
Splaiul Independentei 313,

Bucharest 060042, Romania;
e-mail: crisghiu@yahoo.com

\vspace{0.1 cm}
$^2$University Politehnica of Bucharest,
Faculty of Applied Sciences, Department of Mathematics-Informatics,
Splaiul Independentei 313,
Bucharest 060042, Romania;
e-mails: ralucacoada@yahoo.com; udriste@mathem.pub.ro; vascatevy@yahoo.fr

}
\end{center}

\begin{abstract}
Our original results refer to multivariate recurrences: discrete multitime diagonal recurrence,
bivariate recurrence, trivariate recurrence, solutions tailored to particular situations,
second order multivariate recurrences, characteristic equation, and multivariate diagonal recurrences of superior order.
We find the solutions, we clarify the structural background and provides short,
conceptual proofs. The original results include a new point of view on discrete minimal submanifolds.
\end{abstract}



\section{Discrete multitime recurrences}

The theory of multi-variate recurrences is a current effervescent topic in mathematics today.
These recurrences are based on multiple sequences and come from areas like analysis of algorithms,
computational biology, information theory, queueing theory, filters theory, statistical physics etc.

We consider the lattice of points with integer positive coordinates in $\mathbb{R}^n$.
A multi-variate recurrence is a set of rules which transfer a point into another, together
with initial conditions, capable to cover the hole lattice.

A linear multivariate recurrence with polynomial coefficients corresponds
to a linear PDE. In addition, extending the division to
the context of differential operators, the case of recurrences with polynomial coefficients
can be treated in an analogous way.

Bousquet-M\' elou and Petkov\v sek \cite{BP} analyse
the multivariate linear recurrences with constant coefficients
(see also \cite{E}).

Analyzing linear image processing (representations of filters), Roesser \cite{Ro} used
a class of linear dynamical systems in two discrete-time variables,
$$x(i+1,j)= A_1x(i,j)+A_2 y(i,j),\,\, y(i,j+1)= A_3x(i,j)+A_4y(i,j),$$
which can be extended as block matrix system
$$\begin{pmatrix} x_1(t^1+1,t^2,...,t^m) \\ x_2(t^1,t^2+1,...,t^m)\\ ... \,\,...\,\, ...\\x_N(t^1,t^2,...,t^m+1)\end{pmatrix}=
\begin{pmatrix}A_{11}(t)&A_{12}(t)&...&A_{1N}(t)\\ A_{21}(t)&A_{22}(t)&...&A_{2N}(t)\\ ...& ...&...&...\\ A_{N1}(t)& A_{N2}(t)&...&A_{NN}(t) \end{pmatrix}
\begin{pmatrix}x_1(t)\\ x_2(t)\\ ... \,\,... \,\,...\\ x_N(t)\end{pmatrix},$$
where $x(t)= (x_1(t),..., x_N(t))$.

Fornasini and Marchesini \cite{FM}-\cite{FMa} introduced another class of linear dynamical systems in
two discrete-time variables,
$$x(i + 1, j + 1) = Ax(i + 1, j) + Bx(i, j + 1) + Cx(i, j ),$$
which can be extended as
$$x(t+{\bf 1}) = B_0 \,x(t) + \sum_{\alpha=1}^m B_\alpha \,x(t+1_\alpha).$$
This model is used for image processing, representation of discretized
partial differential equations, models of different physical phenomena,
single-carriage way traffic flow, and river pollution.
Many other authors \cite{Ah}, \cite{DM}, \cite{Hu}, \cite{JL}-\cite{MA} develop the theory of filters.
Both of the foregoing extensions are connected to $mD$ filters theory.

Some interesting related works are those of Prepeli\c ta \cite{CP}, \cite{Pr}, where a multiple hybrid
Laplace and $z$ type transformation was introduced to solve multiple differential-difference and
multiple integral equations and to obtain the frequency-domain representations of
multidimensional hybrid control systems.

The visual inspection of (i) bivariate or (ii) trivariate recurrence plots reveals some typical
geometrical structures: (i) single dots, diagonal lines as well as vertical and
horizontal lines (the combination of vertical and horizontal lines plainly forms
rectangular clusters of recurrence points); (ii) all the above and planes.

The papers \cite{C}, \cite{PP} presents algorithms to compute stable discrete minimal surfaces.
We add a new point of view (coming from \cite{U7}, \cite{UAC}) in studying this subject.

In the class of multivariate sequences $x(t)=x(t^1,...,t^m)$ we distinguish:

{\bf Separable multivariate sequences} (those which can be written like a product): $x(t)=x^1(t^1)\cdots x^m(t^m)$.

{\bf Multi-periodic sequences}:
$$x(t^1,...,t^m)= x(t^1+k_1T^1,...,t^m+k_m T^m),\,\, \forall (k_1,...,k_m)\in \mathbb{Z}^m,$$
equivalently to the fact that every point $(0,\ldots,0,t^\beta,0,\ldots,0)$ is a period.

{\bf Diagonal-periodic sequences}:
$$x(t^1,...,t^m)=x(t^1+T^1,...,t^m+T^m),$$
where $T=(T^1,...,T^m)$ is a vector-period.

\section{Discrete multitime diagonal recurrence}

Any element $t=(t^1,\ldots,t^m) \in \mathbb{N}^m$ is called {\it discrete multitime}.
A function of the type $x:\mathbb{N}^m \to \mathbb{R}^n$ is called {\it multivariate sequence}.

Let $F \colon \mathbb{N}^m \times \mathbb{R}^n \to \mathbb{R}^n$ and ${\bf 1}=(1, \ldots , 1) \in  \mathbb{N}^m$.
We shall study {\it multivariate diagonal recurrences} or
({\it discrete multitime diagonal finite difference equations}) of first order
\begin{equation}
\label{recGeneral}
x(t+{\bf 1})= F(t,x(t)),
\end{equation}
where the multivariate vector sequence $x(t)=(x^1(t),\ldots,x^n(t))$
is solution of the system ({\rm \ref{recGeneral}}).

This model of diagonal recurrence can be justified by the fact that to a first order PDE system
$$D_{\bf 1}x(t)= f(t,x(t)), \, t \in \mathbb{R}^m$$
we can associate a discretized equation of the form
$$x(t+{\bf 1})-x(t) = f(t,x(t)),\, t\in \mathbb{Z}^m.$$
The initial (Cauchy) conditions on a curve, for the PDE system, are translated
into initial conditions for the diagonal recurrence.

\begin{remark} Let us consider a two-variate recurrence with the unknown sequence $x(m,n)$.
A very frequent and interesting case happens when the difference $m-n$, between
the arguments of the unknown sequence, is constant among all its occurrences in the
multivariate recurrence relation. For instance, this happens for any diagonal recurrence
$x(m,n) = f(x(m-1,n-1))$, where the difference between the first and second
argument of $x(m,n)$ is always $m-n$. Such a recurrence can be rewritten as a univariate recurrence
$y(t) = f(y(t-1))$, where $y(t-k) = x(m-k,n-k), \forall \gamma \in \mathbb{N}$.
Another interesting case, similar to the one above, is when the sum $m+n$ of the arguments of
the unknown sequence is constant. For instance, multivariate recurrences
of the form $x(m,n) = f(x(m+1,n-1))$ can be rewritten as univariate recurrences $y(t) = f(y(t-1))$,
where $y(t-k) = x(m+k,n-k), \forall k \in \mathbb{N}$.
\end{remark}

\subsection{Linear discrete single-time recurrence}

Let us recall a well-known result regarding a single-time linear recurrence equation, i.e., $m=1$,
in an original version that can be extended to the multi-temporal case.
\begin{proposition}
\label{rec.p1}
Let $A \colon \mathbb{N} \to \mathcal{M}_{n}(\mathbb{R})$,
$b \colon \mathbb{N} \to \mathbb{R}^n$ and $x_0 \in \mathbb{R}^n$.
Then the unique sequence $x \colon \mathbb{N} \to \mathbb{R}^n$
which verifies the first order linear single-time recurrence equation
\begin{equation}
\label{ecp1.1}
\begin{split}
x(t+1)
&=
A(t)x(t)+b(t),
\quad \forall t \in \mathbb{N},\\
x(0)
&=x_0
\end{split}
\end{equation}
is
\begin{equation}
\label{ecp1.2}
\begin{split}
&x(t)=
A(t-1)A(t-2)\cdot \ldots \cdot A(1)A(0)x_0+\\
+&b(t-1)+
\sum_{k=0}^{t-2}
A(t-1)A(t-2)\cdot \ldots \cdot A(k+1)b(k),
\quad \forall t \geq 2,\\
&x(1)
=A(0)x_0+b(0).
\end{split}
\end{equation}
\end{proposition}

\begin{proof}
Mathematical induction after $t \geq 1$.
For $t=1$, the result is verified automatically,
using the first relation in ({\rm \ref{ecp1.1}}),
for $t=0$, as well as $x(0)=x_0$.
For $t=2$, $x(2)=A(1)x(1)+b(1)=A(1)A(0)x_0+A(1)b(0)+b(1)$.

It remains to solve the inductive step:
prove that, if the statement holds for some natural number $t\geq 2$,
then the statement holds for $t + 1$;
indeed,
\begin{equation*}
\begin{split}
&\quad x(t+1)
=
A(t)x(t)+b(t)\\
=&
A(t)A(t-1)A(t-2)\cdot \ldots \cdot A(1)A(0)x_0+\\
+&A(t)b(t-1)+
\sum_{k=0}^{t-2}
A(t)A(t-1)A(t-2)\cdot \ldots \cdot A(k+1)b(k)
+b(t)\\
=&
A(t)A(t-1)\cdot \ldots \cdot A(1)A(0)x_0
+b(t)+
\sum_{k=0}^{t-1}
A(t)A(t-1)\cdot \ldots \cdot A(k+1)b(k).
\end{split}
\end{equation*}
\end{proof}

\begin{corollary}
Let $A \in \mathcal{M}_{n}(\mathbb{R})$
be a constant matrix, and $x_0 \in \mathbb{R}^n$.
Then the unique sequence $x \colon \mathbb{N} \to \mathbb{R}^n$ which verifies
\begin{equation}
\label{ecCor1.1}
\begin{split}
x(t+1)
&=
Ax(t)+b(t),
\quad \forall t \in \mathbb{N},\\
x(0)
&=x_0,
\end{split}
\end{equation}
is
\begin{equation}
\label{ecCor1.2}
x(t)=
A^tx_0+
\sum_{k=0}^{t-1}
A^{t-1-k}\,b(k),
\quad \forall t \geq 1.
\end{equation}
\end{corollary}

\subsection{Linear discrete multitime diagonal recurrence}

The diagonal discretization of PDEs incorporate points which lie on diagonals of the grid.
This is a good enough reason for introducing and analyzing the diagonal recurrences.

Let $m\geq 2$,
$A \colon \mathbb{N}^m \to \mathcal{M}_{n}(\mathbb{R})$,
$b \colon \mathbb{N}^m \to \mathbb{R}^n$.
In this Subsection we refer to a
{\it linear discrete multitime diagonal recurrence equation of first order}
\begin{equation}
\label{eeec}
x(t+{\bf 1})=A(t)x(t)+b(t),
\quad t \in \mathbb{N}^m,
\end{equation}
with $x \colon \mathbb{N}^m \to \mathbb{R}^n=\mathcal{M}_{n,1}(\mathbb{R})$.

For convenience, we denote $\mu(t)=\min \{t^1,t^2,\ldots, t^m\}$.

\begin{lemma}
\label{rec.lem1}
Let $D= \Big\{\, t=(t^1,t^2, \dots , t^m ) \in \mathbb{N}^m\,  \Big|\,
\mu(t)=t^m \, \Big\}$, $m\geq 2$, and \\
$A \colon D \to \mathcal{M}_{n}(\mathbb{R})$,
$b \colon D \to \mathbb{R}^n$,
$f \colon \mathbb{N}^{m-1} \to \mathbb{R}^n$.
Then the unique multivariate sequence
$x \colon D \to \mathbb{R}^n$ which verifies the first order linear
discrete multitime diagonal recurrence equation
\begin{equation}
\label{ecl1.1}
\begin{split}
x(t+{\bf 1})
&=
A(t)x(t)+b(t),
\quad \forall t \in D,\\
x(t^1, \dots , t^{m-1},& 0)
=f(t^1, \dots , t^{m-1}),
\quad \forall (t^1, \dots , t^{m-1}) \in \mathbb{N}^{m-1}
\end{split}
\end{equation}
is
\begin{equation}
\label{ecl1.2}
\begin{split}
x(t)=
A(t-{\bf 1})A(t-2\cdot {\bf 1})\cdot \ldots
\cdot &A(t-t^m \cdot  {\bf 1})f(t^1-t^m, \dots , t^{m-1}-t^m)\\
+b(t-{\bf 1})+
\sum_{k=2}^{t^m}
A(t-{\bf 1})A(t-2\cdot & {\bf 1})\cdot
\ldots \cdot A(t-(k-1)\cdot {\bf 1})b(t-k\cdot {\bf 1}),\\
\forall t\in D,\,\,\mbox{ with }\, t^m \geq 2,
\end{split}
\end{equation}

\begin{equation}
\label{ecl1.3}
\begin{split}
&x(t^1, \dots , t^{m-1}, 1)
=
A(t^1-1, \dots , t^{m-1}-1, 0)
f(t^1-1, \dots , t^{m-1}-1)\\
&+b(t^1-1, \dots , t^{m-1}-1, 0),
\quad \forall (t^1, \dots , t^{m-1}) \in (\mathbb{N^*})^{m-1}.
\end{split}
\end{equation}
\end{lemma}

\begin{proof}
Let us remark that for any
$t=(t^1,t^2,\ldots,t^m)\in D$ and any $s \in \mathbb{N}$, it follows
$t-(t^m-s)\cdot {\bf 1}=
(t^1-t^m+s, \ldots, t^{m-1}-t^m+s, s)\in D$.
Hence, for a fixed $t \in D$, we can define the sequences
$$
\widetilde{A} \colon \mathbb{N} \to \mathcal{M}_{n}(\mathbb{R}),
\quad
\widetilde{b} \colon \mathbb{N} \to \mathbb{R}^n,
\quad
\widetilde{x} \colon \mathbb{N} \to \mathbb{R}^n,
$$
\begin{equation}
\label{ecl1.4}
\widetilde{A}(s)=A(t-(t^m-s)\cdot {\bf 1})
=
A(t^1-t^m+s, \ldots, t^{m-1}-t^m+s, s),
\end{equation}
\begin{equation}
\label{ecl1.5}
\widetilde{b}(s)=b(t-(t^m-s)\cdot {\bf 1})
=
b(t^1-t^m+s, \ldots, t^{m-1}-t^m+s, s),
\end{equation}
\begin{equation}
\label{ecl1.6}
\widetilde{x}(s)=x(t-(t^m-s)\cdot {\bf 1})
=
x(t^1-t^m+s, \ldots, t^{m-1}-t^m+s, s).
\end{equation}

The sequence $\widetilde{x}$ verifies
\begin{equation}
\label{ecl1.7}
\begin{split}
\widetilde{x}(s+1)
&=
\widetilde{A}(s)\widetilde{x}(s)+\widetilde{b}(s),
\quad \forall s \in \mathbb{N},\\
\widetilde{x}(0)
&=f(t^1-t^m, \dots , t^{m-1}-t^m),
\end{split}
\end{equation}

Now we can apply the Proposition \ref{rec.p1}.
For any $s\geq 2$, we have
\begin{equation*}
\begin{split}
\widetilde{x}(s)=
\widetilde{A}(s-1)\widetilde{A}(s-2)\cdot
\ldots \cdot \widetilde{A}(0)
\widetilde{x}(0)+\\
+\widetilde{b}(s-1)+
\sum_{k=0}^{s-2}
\widetilde{A}(s-1)\widetilde{A}(s-2)\cdot
\ldots \cdot \widetilde{A}(k+1)\widetilde{b}(k)
\end{split}
\end{equation*}
i.e.,
\begin{equation*}
x(t-(t^m-s)\cdot {\bf 1})
=
A(t-(t^m-s+1)\cdot {\bf 1})A(t-(t^m-s+2)\cdot {\bf 1})\cdot
\ldots \cdot A(t-t^m\cdot {\bf 1})
\widetilde{x}(0)+
\end{equation*}
\begin{equation*}
+b(t-(t^m-s+1)\cdot {\bf 1})
\end{equation*}
\begin{equation*}
+
\sum_{k=0}^{s-2}
A(t-(t^m-s+1)\cdot {\bf 1})A(t-(t^m-s+2)\cdot {\bf 1})\cdot
\ldots \cdot
A(t-(t^m-k-1)\cdot {\bf 1})
b(t-(t^m-k)\cdot {\bf 1})
\end{equation*}
For $t^m\geq 2$, we set $s=t^m$ in the foregoing relation. Then
\begin{equation*}
x(t)
=
A(t- {\bf 1})A(t-2\cdot {\bf 1})\cdot
\ldots \cdot A(t-t^m\cdot {\bf 1})
\widetilde{x}(0)
+b(t- {\bf 1})
\end{equation*}
\begin{equation*}
+
\sum_{k=0}^{t^m-2}
A(t- {\bf 1})A(t-2\cdot {\bf 1})\cdot
\ldots \cdot
A(t-(t^m-k-1)\cdot {\bf 1})
b(t-(t^m-k)\cdot {\bf 1})
\end{equation*}
\begin{equation*}
=
A(t- {\bf 1})A(t-2\cdot {\bf 1})\cdot
\ldots \cdot A(t-t^m\cdot {\bf 1})
f(t^1-t^m, \dots , t^{m-1}-t^m)
+b(t- {\bf 1})
\end{equation*}
\begin{equation*}
+
\sum_{k=2}^{t^m}
A(t- {\bf 1})A(t-2\cdot {\bf 1})\cdot
\ldots \cdot
A(t-(t^m-k+1)\cdot {\bf 1})
b(t-(t^m-k+2)\cdot {\bf 1})
\end{equation*}
\begin{equation*}
=
A(t- {\bf 1})A(t-2\cdot {\bf 1})\cdot
\ldots \cdot A(t-t^m\cdot {\bf 1})
f(t^1-t^m, \dots , t^{m-1}-t^m)
\end{equation*}
\begin{equation*}
+b(t- {\bf 1})
+
\sum_{k=2}^{t^m}
A(t- {\bf 1})A(t-2\cdot {\bf 1})\cdot
\ldots \cdot
A(t-(k-1)\cdot {\bf 1})
b(t-k\cdot {\bf 1})
\end{equation*}

For $t^m=1$, we obtain the formula
(\ref{ecl1.3}):
\begin{equation*}
x(t^1, \dots , t^{m-1}, 1)
=
x((t^1-1, \dots , t^{m-1}-1, 0)+{\bf 1})
=
\end{equation*}
\begin{equation*}
=
A(t^1-1, \dots , t^{m-1}-1, 0)
x(t^1-1, \dots , t^{m-1}-1, 0)
+b(t^1-1, \dots , t^{m-1}-1, 0)
\end{equation*}
\begin{equation*}
=
A(t^1-1, \dots , t^{m-1}-1, 0)
f(t^1-1, \dots , t^{m-1}-1)
+b(t^1-1, \dots , t^{m-1}-1, 0).
\end{equation*}
\end{proof}

\begin{remark}
\label{rec.obs1}
In the conditions of Lemma {\rm \ref{rec.lem1}},
if moreover the function $A(\cdot)$ is constant,
i.e., $A(t)=A$, $\forall t$, then
the formulas {\rm (\ref{ecl1.2})}
and {\rm (\ref{ecl1.3})} become
\begin{equation}
\label{ecObs1.1}
x(t)=
A^{t^m}f(t^1-t^m, \dots , t^{m-1}-t^m)
+
\sum_{k=1}^{t^m}
A^{k-1}b(t-k\cdot {\bf 1}),
\end{equation}
\begin{equation*}
\forall t\in D,\,\,\mbox{ with }\, t^m \geq 1.
\end{equation*}
\end{remark}

Notation: for every $\beta \in \{1, 2, \ldots, m\}$, we denote
$$(t^1,\ldots, \widehat{t^\beta}, \ldots, t^m):=
(t^1,\ldots, t^{\beta-1},t^{\beta+1} \ldots, t^m)\in \mathbb{N}^{m-1}.$$

\begin{theorem}
\label{rec.t1}
Let $m\geq 2$,
$A \colon \mathbb{N}^m \to \mathcal{M}_{n}(\mathbb{R})$,
$b \colon \mathbb{N}^m \to \mathbb{R}^n$. We consider the $(m-1)$-sequences
$f_1, f_2, \ldots, f_m \colon \mathbb{N}^{m-1} \to \mathbb{R}^n$,
such that
\begin{equation}
\label{ect1.1}
\begin{split}
f_{\alpha}(t^1,\ldots, \widehat{t^\alpha}, \ldots, t^m)\Big|_{t^\beta=0}
&=
f_{\beta}(t^1,\ldots, \widehat{t^\beta}, \ldots, t^m)\Big|_{t^\alpha=0},\\
\forall
t^1, \ldots, t^{\alpha-1}, t^{\alpha+1}, &\ldots, t^{\beta-1}, t^{\beta+1}, \ldots, t^m
\in \mathbb{N},
\end{split}
\end{equation}
for any $\alpha,\beta \in \{1, 2, \ldots, m\}$.
Then the unique $m$-sequence
$x \colon \mathbb{N}^m \to \mathbb{R}^n$ which verifies
\begin{equation}
\label{ect1.2}
\left\{
\begin{split}
x(t+{\bf 1})
&=
A(t)x(t)+b(t),
\quad \forall t \in \mathbb{N}^m,\\
x(t)\Big|_{t^\beta=0}
&=
f_{\beta}(t^1,\ldots, \widehat{t^\beta}, \ldots, t^m),
\quad
\forall
(t^1,\ldots, \widehat{t^\beta}, \ldots, t^m)\in \mathbb{N}^{m-1},\\
\forall
\beta &\in \{1,2, \ldots, m\},
\end{split}
\right.
\end{equation}
is defined either by the formula
\begin{equation}
\label{ect1.3}
\begin{split}
&x(t)=
A(t-{\bf 1})A(t-2\cdot {\bf 1})\cdot \ldots
\cdot A(t-t^\beta \cdot  {\bf 1})\cdot\\
& \qquad \cdot
f_{\beta}(t^1-t^\beta, \ldots, \widehat{t^\beta}, \ldots, t^{m-1}-t^\beta)
+b(t-{\bf 1})\\
&+
\sum_{k=2}^{t^\beta}
A(t-{\bf 1})A(t-2\cdot  {\bf 1})\cdot
\ldots \cdot A(t-(k-1)\cdot {\bf 1})b(t-k\cdot {\bf 1}),\\
&\qquad \mbox{ if }\,\,
\mu(t)=t^\beta \geq 2,
\end{split}
\end{equation}
or by the formula
\begin{equation}
\label{ect1.4}
\begin{split}
x(t)
&=
A(t^1-1, \dots ,t^{\beta-1}-1,0,t^{\beta+1}-1, \ldots t^m-1)
\cdot
\\
& \,\, \cdot
f_{\beta}(t^1-1, \ldots, \widehat{t^\beta}, \ldots, t^{m-1}-1)\\
&+b(t^1-1, \dots ,t^{\beta-1}-1,0,t^{\beta+1}-1, \ldots t^m-1),\\
&\quad \quad  \mbox{ if }\,\,
\mu(t)=t^\beta=1.
\end{split}
\end{equation}
\end{theorem}

\begin{proof}
Let us remark that the multivariate sequence $x$ is well defined,
i.e., if $t^\alpha=t^\beta=\mu(t)$, then the expressions
which define $x$ of the formulas (\ref{ect1.3})
(or (\ref{ect1.4})), corresponding to $\alpha$ and $\beta$,
coincide due to the equalities
(\ref{ect1.1}).

If $t^\beta=\mu(t)$,
the conclusion follows applying directly the Lemma \ref{rec.lem1},
having $t^\beta$ instead of $t^m$.

Conversely, one observes immediately that if the multivariate sequence $x$
is defined by the formulas (\ref{ect1.3}), (\ref{ect1.4}),
together with the second formula of (\ref{ect1.2}) (for $t^\beta=0$), then the multivariate sequence
$x$ verifies also the first relation in (\ref{ect1.2}), i.e.,
$$
x(t+{\bf 1})
=
A(t)x(t)+b(t),
\quad \forall t \in \mathbb{N}^m.
$$
\end{proof}

\begin{remark}
\label{rec.obs2}
Let $m\geq 2$,
$A \colon \mathbb{N}^m \to \mathcal{M}_{n}(\mathbb{R})$,
$b \colon \mathbb{N}^m \to \mathbb{R}^n$
and the $(m-1)$-sequences
$f_1, f_2, \ldots, f_m \colon \mathbb{N}^{m-1} \to \mathbb{R}^n$.
If the $m$-sequence
$x \colon \mathbb{N}^m \to \mathbb{R}^n$ verifies the relations
\begin{equation}
\label{ecObs2.1}
\begin{split}
x(t)\Big|_{t^\beta=0}
&=
f_{\beta}(t^1,\ldots, \widehat{t^\beta}, \ldots, t^m),
\quad
\forall
(t^1,\ldots, \widehat{t^\beta}, \ldots, t^m)\in \mathbb{N}^{m-1},\\
\forall \beta&\in \{1,2, \ldots, m\},
\end{split}
\end{equation}
then,
\begin{equation}
\label{ecObs2.2}
\begin{split}
f_{\alpha}(t^1,\ldots, \widehat{t^\alpha}, \ldots, t^m)\Big|_{t^\beta=0}
=
f_{\beta}(t^1,\ldots, \widehat{t^\beta}, \ldots, t^m)\Big|_{t^\alpha=0},\\
\forall
t^1, \ldots, t^{\alpha-1}, t^{\alpha+1}, \ldots, t^{\beta-1}, t^{\beta+1}, \ldots, t^m
\in \mathbb{N}.
\end{split}
\end{equation}
\end{remark}

This follows immediately since,
$\displaystyle
x(t)\Big|_{t^\beta=0, t^\alpha=0}
=x(t)\Big|_{t^\alpha=0, t^\beta=0}$.

\begin{corollary}
In the conditions in Theorem {\rm \ref{rec.t1}},
if moreover, the function $A(\cdot)$ is constant,
i.e., $A(t)=A$, $\forall t$, then
the formulas {\rm (\ref{ect1.3})}
and {\rm (\ref{ect1.4})} become
\begin{equation}
\label{consecT1}
\begin{split}
&x(t)=
A^{t^\beta}f_{\beta}(t^1-t^\beta, \ldots, \widehat{t^\beta}, \ldots, t^{m-1}-t^\beta)
+
\sum_{k=1}^{t^\beta}
A^{k-1}b(t-k\cdot {\bf 1}),\\
&\mbox{ if }\,\,
\mu(t)=t^\beta \geq 1.
\end{split}
\end{equation}
\end{corollary}

Identifying the initial conditions with the constant diagonal recurrence, we obtain new information
about the discrete diagonal flow.

\begin{proposition}
\label{rec.p2}
Let $m\geq2$ and $A \colon \mathbb{N}^m \to \mathcal{M}_{n}(\mathbb{R})$.
Denote
\begin{equation*}
S=
\Big\{ y \colon \mathbb{N}^m \to \mathbb{R}^n \, \Big|\,
y(t+{\bf 1})=y(t), \forall t \in \mathbb{N}^m\, \Big\},
\end{equation*}
\begin{equation*}
V=
\Big\{ x \colon \mathbb{N}^m \to \mathbb{R}^n \, \Big|\,
x(t+{\bf 1})=A(t)x(t), \forall t \in \mathbb{N}^m\, \Big\}
\end{equation*}
and introduce the function
$$\psi \colon S \to V,$$
\begin{equation}
\label{ecp2.1}
\psi \big( y(\cdot) \big )(t)
=
\begin{cases}
\Bigg(
\displaystyle \prod_{k=1}^{\mu(t)}
A(t-k \cdot  {\bf 1})
\Bigg)
\cdot y \big( t-\mu(t) \cdot  {\bf 1} \big),\,\, \mbox{ if }\,\,
\mu(t)\geq 1,\\
\qquad \qquad \quad \quad
y(t),\,\,\qquad \qquad \quad \qquad\,
\mbox{ if }\,\,
\mu(t)=0.
\end{cases}
\end{equation}

$a)$ The sets $S$ and $V$ are real vector spaces, and
$\psi$ is an isomorphism of vector spaces.

$b)$ The vector space $V$ has infinite dimension.
\end{proposition}

\begin{proof}
$a)$
First, let  us observe that the application $\psi$ is well defined,
i.e.,  the $m$-sequences definite by the formula
(\ref{ecp2.1}) verify the recurrence
$x(t+{\bf 1})=A(t)x(t)$. This follows immediately from the Theorem \ref{rec.t1}.

The respective sequence is the unique sequence
$x \colon \mathbb{N}^m \to \mathbb{R}^n$ which verifies the problem
\begin{equation}
\label{ecp2.3}
\left\{
\begin{split}
x(t+{\bf 1})
&=
A(t)x(t),
\quad \forall t \in \mathbb{N}^m,\\
x(t)\Big|_{t^\beta=0}
&=
y(t^1,\ldots, t^{\beta-1},0,t^{\beta+1}, \ldots, t^m),
\,\,
\forall
(t^1,\ldots, \widehat{t^\beta}, \ldots, t^m)\in \mathbb{N}^{m-1},\\
\forall \beta&\in \{1,2, \ldots, m\}.
\end{split}
\right.
\end{equation}

Obviously, $V$ and $S$ are vector spaces over $\mathbb{R}$.

One observes that $\psi$ is a morphism of vector spaces.
Let us prove that the function $\psi$ is injective:
let $y \in S$ such that
$\psi \big( y(\cdot) \big )(t)=0$, $\forall t \in \mathbb{N}^m$.
It follows $\psi \big( y(\cdot) \big )(t)\Big|_{t^\beta=0}=0$;
but, according to the relation (\ref{ecp2.1}), we have
$\psi \big( y(\cdot) \big )(t)\Big|_{t^\beta=0}=
y(t)\Big|_{t^\beta=0}$ and we obtain
$y(t)\Big|_{t^\beta=0}=0$.
Hence $y$ is the unique $m$-sequence which verifies the problem
\begin{equation*}
\left\{
\begin{split}
y(t+{\bf 1})
&=
y(t),
\quad \forall t \in \mathbb{N}^m,\\
y(t)\Big|_{t^\beta=0}
&=
0,
\,\,
\forall
(t^1,\ldots, \widehat{t^\beta}, \ldots, t^m)\in \mathbb{N}^{m-1},\\
\forall \beta&\in \{1,2, \ldots, m\}.
\end{split}
\right.
\end{equation*}
Applying the Theorem \ref{rec.t1}, it follows
$y(t)=0$, $\forall t \in \mathbb{N}^m$, i.e.,
$y$ is $m$-sequence zero.

The surjectivity of $\psi$: let $x \in V$.
We choose $y \colon \mathbb{N}^m \to \mathbb{R}^n$,
the unique $m$-sequence which verifies
\begin{equation*}
\left\{
\begin{split}
y(t+{\bf 1})
&=
y(t),
\quad \forall t \in \mathbb{N}^m,\\
y(t)\Big|_{t^\beta=0}
&=
x(t^1,\ldots, t^{\beta-1},0,t^{\beta +1}, \ldots, t^m),
\,\,
\forall
(t^1,\ldots, \widehat{t^\beta}, \ldots, t^m)\in \mathbb{N}^{m-1},\\
\forall \beta&\in \{1,2, \ldots, m\}.
\end{split}
\right.
\end{equation*}
Obviously $y \in S$ and $x(t)\Big|_{t^\beta=0}=y(t)\Big|_{t^\beta=0}$.
Hence the $m$-sequence $x$ verifies the relations
(\ref{ecp2.3}). It follows (Theorem \ref{rec.t1}) that
\begin{equation*}
x(t)
=
A(t-{\bf 1})A(t-2\cdot {\bf 1})\cdot \ldots
\cdot A(t-t^\beta \cdot  {\bf 1})\cdot y(t-t^\beta \cdot  {\bf 1}),
\mbox{ if }\,
\mu(t)=t^\beta \geq 1;
\end{equation*}
\begin{equation*}
x(t)
=y(t),\,\,
\mbox{ if }\,\,
\mu(t)=0,
\end{equation*}
hence, according the relation (\ref{ecp2.1}),
we have
$\psi \big( y(\cdot) \big )(t)=x(t)$, $\forall t \in \mathbb{N}^m$,
i.e., $\psi \big( y(\cdot) \big )=x(\cdot)$.

$b)$ Since $V$ and $S$ are isomorphic vector spaces, it is
sufficient to show that $S$ has an infinite dimension.
Equivalently, we shall show that $S$ contains an infinity of
linearly independent elements.

Let $v\in \mathbb{R}^n$, $v \neq 0$.
For each $k \in \mathbb{N}^*$, we consider the sequence
\begin{equation*}
y_k \colon \mathbb{N}^m \to \mathbb{R}^n,\,\,\,
y_k(t)=\big(  t^1 -t^2  \big )^kv,
\,\,\,
\forall t=(t^1,t^2, \ldots, t^m)\in \mathbb{N}^m.
\end{equation*}
Since
$y_k(t+{\bf 1})=\big(  t^1+1 -t^2-1  \big )^kv=y_k(t)$,
it follows $y_k \in S$.

Let $F \subseteq \mathbb{N}^*$, $F$ finite and non-void.
For each $k \in F$, we consider the sequence $a_k \in \mathbb{R}$,
such that
$\displaystyle
\sum_{k \in F} a_k y_k(\cdot)=0
$,
i.e.,
$
\displaystyle
\sum_{k \in F} a_k y_k(t)=0,
\forall
t\in \mathbb{N}^m.
$
Setting $t_2=0$, it follows
$
\displaystyle
\sum_{k \in F} a_k(t^1 )^kv=0,
\forall
t^1 \in \mathbb{N},
$
or
$
\displaystyle
\bigg( \sum_{k \in F} a_k(t^1 )^k\bigg)v=0,
\forall
t^1 \in \mathbb{N}.
$
Since $v\neq 0$, we deduce that
$
\displaystyle
 \sum_{k \in F} a_k(t^1 )^k=0,
\forall
t^1 \in \mathbb{N}.
$

Consequently the polynomial $P(X):=\displaystyle \sum_{k \in F} a_k X^k$
vanishes for any $t^1 \in \mathbb{N}$, i.e.,
$P(X)$ has an infinity of roots; whence, it follows that $P(X)$ is zero polynomial, i.e.,
$a_k=0$, $\forall k \in F$.

We proved that
$\big\{ y_k\, \big|\, k \in \mathbb{N}^* \big\}$
is a subset of $S$ consisting in linear independent elements.
Hence the dimension of $S$ is infinite.
\end{proof}

\begin{remark}
\label{rec.obs3}
For determining the solutions of the recurrence
$$x(t+{\bf 1})=A(t)x(t), \quad t \in \mathbb{N}^m,$$
it is sufficient to know the solutions of the recurrence
$$y(t+{\bf 1})=y(t), \quad t \in \mathbb{N}^m.$$
\end{remark}

Suppose that the functions
$f_1, f_2, \ldots, f_m \colon \mathbb{N}^{m-1} \to \mathbb{R}^n$
satisfy the relations {\rm (\ref{ect1.1})}. We want to determine the multivariate sequence
$x \colon \mathbb{N}^m \to \mathbb{R}^n$ such that
\begin{equation*}
\left\{
\begin{split}
x(t+{\bf 1})
&=
A(t)x(t),
\quad \forall t \in \mathbb{N}^m,\\
x(t)\Big|_{t^\beta=0}
&=
f_{\beta}(t^1,\ldots, \widehat{t^\beta}, \ldots, t^m),
\quad
\forall
(t^1,\ldots, \widehat{t^\beta}, \ldots, t^m)\in \mathbb{N}^{m-1},\\
\forall \beta&\in \{1,2, \ldots, m\}.
\end{split}
\right.
\end{equation*}

If $y \colon \mathbb{N}^m \to \mathbb{R}^n$ verifies
\begin{equation*}
\left\{
\begin{split}
y(t+{\bf 1})
&=
y(t),
\quad \forall t \in \mathbb{N}^m,\\
y(t)\Big|_{t^\beta=0}
&=
f_{\beta}(t^1,\ldots, \widehat{t^\beta}, \ldots, t^m),
\quad
\forall
(t^1,\ldots, \widehat{t^\beta}, \ldots, t^m)\in \mathbb{N}^{m-1},\\
\forall \beta &\in \{1,2, \ldots, m\},
\end{split}
\right.
\end{equation*}
then, from the proof of the Proposition {\rm \ref{rec.p2}}, it follows
$x(\cdot)=\psi \big( y(\cdot) \big )$.

\subsection{Examples}

{\bf Bivariate recurrences} Let $A \in \mathcal{M}_{n}(\mathbb{R})$
and two sequences
$f_1, f_2 \colon \mathbb{N} \to \mathbb{R}^n$,
with $f_1(0)=f_2(0)$. Then the unique double sequence
$x \colon \mathbb{N}^2 \to \mathbb{R}^n$ which solves the problem
\begin{equation*}
\left\{
\begin{split}
x(t^1+1,t^2+1)
&=
A\cdot x(t^1,t^2),
\quad \forall (t^1,t^2) \in \mathbb{N}^2,\\
x(0,t^2)
&=
f_1(t^2),
\quad
\forall
t^2 \in \mathbb{N},\\
x(t^1,0)
&=
f_2(t^1),
\quad
\forall
t^1 \in \mathbb{N}
\end{split}
\right.
\end{equation*}
is
\begin{equation*}
x(t^1,t^2)=
\left\{
  \begin{array}{ll}
    A^{t^1}f_1(t^2-t^1), & \hbox{ if }\,\, t^1\leq t^2 \\
    A^{t^2}f_2(t^1-t^2), & \hbox{ if } \,\,t^2\leq t^1.
  \end{array}
\right.
\end{equation*}

{\bf Trivariate recurrences} Let $A \in \mathcal{M}_{n}(\mathbb{R})$
and three double sequences
$f_1, f_2, f_3 \colon \mathbb{N}^2 \to \mathbb{R}^n$,
with
$$
f_1(0,t^3)=f_2(0,t^3),\,
f_1(t^2,0)=f_3(0,t^2)
,\,
f_2(t^1,0)=f_3(t^1,0),
\quad
\forall t^1,t^2,t^3\in \mathbb{N}.
$$
Then the unique triple sequence
$x \colon \mathbb{N}^3 \to \mathbb{R}^n$ which solves the problem
\begin{equation*}
\left\{
\begin{split}
x(t^1+1,t^2+1,t^3+1)
&=
A\cdot x(t^1,t^2,t^3),
\quad \forall (t^1,t^2,t^3) \in \mathbb{N}^3,\\
x(0,t^2,t^3)
&=
f_1(t^2,t^3),
\quad
\forall
(t^2,t^3) \in \mathbb{N}^2,\\
x(t^1,0,t^3)
&=
f_2(t^1,t^3),
\quad
\forall
(t^1,t^3) \in \mathbb{N}^2,\\
x(t^1,t^2,0)
&=
f_3(t^1,t^2),
\quad
\forall
(t^1,t^2) \in \mathbb{N}^2
\end{split}
\right.
\end{equation*}
is
\begin{equation*}
x(t^1,t^2,t^3)=
\left\{
  \begin{array}{ll}
    A^{t^1}f_1(t^2-t^1,t^3-t^1), & \hbox{ if }\,\,
t^1=\min\{t^1,t^2,t^3\} \\
    A^{t^2}f_2(t^1-t^2,t^3-t^2), & \hbox{ if }\,\,
t^2=\min\{t^1,t^2,t^3\} \\
    A^{t^3}f_3(t^1-t^3,t^2-t^3), & \hbox{ if }\,\,
t^3=\min\{t^1,t^2,t^3\}.
  \end{array}
\right.
\end{equation*}

\subsection{Solutions tailored to particular situations}

To find the solution for a discrete multitime diagonal recurrence, with constant coefficients,
$x(t+{\bf 1}) = Ax(t)$, we can use a discrete single-time recurrence
$y(t+1)=Ay(t)$, together a family of initial conditions (see the foregoing examples).
Particularly, for any constant matrix $A$, the solution of first order discrete multitime diagonal recurrence can be written as
$$x(t) = A^{<\epsilon,t>}\,x_0, \, \epsilon=(\epsilon_1,...,\epsilon_m), \,<\epsilon,{\bf 1}>=1.$$
The existence conditions of the powers of the matrix $A$ gives the conditions: (i) $\epsilon \in \mathbb{Z}^m$
if $A$ is non-degenerate, (ii) $\epsilon \in \mathbb{N}^m$ if $A$ is degenerate. If we add initial condition,
then, in both cases, one and only one component of $\epsilon$ is non-zero (depending on the initial condition).

\begin{theorem} Let $A\in {\cal M}_n(\mathbb{R})$ be a matrix which verify the equation $A^m=A$. For
any $x_0\in \mathbb{R}^n$ (constant), the vector $x(t)=A^{t^1+...+t^m}x_0$ verify the diagonal recurrence
$x(t+{\bf 1})=Ax(t)$.\end{theorem}

\begin{proof} By computation,
$x(t+{\bf 1})=A^{m+t^1+...+t^m}x_0=A^{m}A^{t^1+...+t^m}x_0=Ax(t)$.
\end{proof}

A matrix $B$ is said to be an $m$-th root of an $n\times n$ matrix $A$ if $B^m = A$, where $m$ is a positive integer greater than or
equal to $2$. If there is no such matrix for any integer $m \geq 2$, then $A$ is called a rootless matrix.
A non-singular matrix and a diagonalizable matrix have $m$-th roots in complex numbers (see \cite{YN}).

\begin{theorem} If the $n\times n$ matrix $A$ has an $m$-root $B$, then a solution of the
diagonal recurrence $x(t+{\bf 1})=Ax(t)$ is $x(t)=B^{t^1+...+t^m}x_0$, where $x_0$ is a constant vector.\end{theorem}

\begin{proof} Explicitly, we have
$x(t+{\bf 1})=B^{m+t^1+...+t^m}x_0=B^{m}B^{t^1+...+t^m}x_0=Ax(t)$.
\end{proof}

\subsection{Characteristic equation}

We refer to the recurrence
$$x(t+{\bf 1})= A x(t),$$
satisfying the initial conditions $(\ref{ect1.2})$ made compatible by $(\ref{ect1.1})$.

If the matrix $A$ is diagonalizable, then there exists a basis $\{v_1,...,v_n\}$
consisting in eigenvectors, in $\mathbb{C}^n=\mathcal{M}_{n,1}(\mathbb{C})$. Denote by $\lambda_k$
the corresponding eigenvalues (distinct or not). Any solution of the previous recurrence
is of the form
$$x(t)=\sum_{k=1}^n c_k(t)\,\lambda_k^{\mu(t)}\,v_k,$$
where $c_k(t)\in \mathbb{C}$ and $c_k(t+{\bf 1})=c_k(t)$.
The functions $c_k(t)$ are determined by the initial conditions $(\ref{ect1.2})$.

\begin{remark}
We look for solutions of the form
$x(t) = v\,\lambda^{<\epsilon,t>}$, with $<\epsilon,{\bf 1}>=1$ and $v\not=0$.
The existence conditions of the powers of the eigenvalue $\lambda$ gives the conditions: (i) $\epsilon \in \mathbb{Z}^m$
if $A$ is non-degenerate, (ii) $\epsilon \in \mathbb{N}^m$ if $A$ is degenerate. It follows
$Av=\lambda v.$ Consequently, $\lambda$ is an eigenvalue, and $v$ is an eigenvector.
The equation $P(\lambda)=\det(A-\lambda I)=0$ is called characteristic equation.
If we can determine $n$ pairs $(\lambda_k,v_k)$, then  a particular solution of the recurrence is
$$x(t)=\sum_{k=1}^n c_k\,\lambda_k^{<\epsilon,t>}\,v_k.$$

When $\epsilon =(0,...,0,1,0,...,0)$, the scalar product is $<\epsilon,t>=t^\alpha$.
We find a solution of the form
$$x(t)=\sum_{k=1}^n c_k\,\lambda_k^{t^\alpha}\,v_k.$$
\end{remark}

\section{Multivariate diagonal recurrences of \\superior order}

Let $k\geq 2$ and the matrix sequences
$$B_{0}, B_1, \ldots, B_{k-1} \colon
\mathbb{N}^m \to \mathcal{M}_{n}(\mathbb{R}),
f \colon \mathbb{N}^m \to \mathbb{R}^n.$$

The linear diagonal vectorial recurrences of order $k$ have the form
\begin{equation}
\label{recOrdSup1}
x(t+k\cdot {\bf 1})
=
\sum_{j=0}^{k-1}B_j(t) x(t+j\cdot {\bf 1})+f(t),
\end{equation}
with $x(t^1, \ldots, t^{\beta-1},0,t^{\beta+1}, \ldots, t^{m})$,
$\ldots$, $x(t^1, \ldots, t^{\beta-1},k-1,t^{\beta+1}, \ldots, t^{m})$
given for any
$\beta \in \{1,2,\ldots, m\}$.
The unknown sequence is
$x \colon \mathbb{N}^m \to \mathbb{R}^n$.

These recurrences easily be reduced to order one recurrences of the form (\ref{ect1.2}).
Indeed, it is enough to consider
$$y \colon \mathbb{N}^m \to \big ( \mathbb{R}^n \big)^k=\mathcal{M}_{nk,1}(\mathbb{R}),\,\,
b \colon \mathbb{N}^m \to \big ( \mathbb{R}^n \big)^k=\mathcal{M}_{nk,1}(\mathbb{R}),$$
\begin{equation}
\label{recOrdSup2}
y(t)=\left(
       \begin{array}{c}
         x(t) \\
         x(t+ {\bf 1}) \\
          x(t+ 2\cdot{\bf 1}) \\
         \vdots \vspace{0.1 cm} \\
         x(t+(k-1)\cdot {\bf 1}) \\
       \end{array}
     \right);
     \quad
b(t)=\left(
       \begin{array}{c}
          O_{n,1} \\
          O_{n,1} \\
         \vdots \vspace{0.1 cm} \\
          O_{n,1} \\
         f(t) \\
       \end{array}
     \right);
\end{equation}
and
$$A \colon \mathbb{N}^m \to  \mathcal{M}_{nk}(\mathbb{R}),$$
\begin{equation}
\label{recOrdSup3}
A(t)=
\left(
  \begin{array}{cccccc}
    O_n & I_n & O_n & \ldots & O_n & O_n \\
    O_n & O_n & I_n & \ldots & O_n & O_n \\
    \vdots \vspace{0.1 cm}  &  &  &  &  \\
    O_n & O_n & O_n & \ldots & O_n & I_n \\
    B_0(t) & B_1(t) & B_2(t) & \ldots & B_{k-2}(t) & B_{k-1}(t) \\
  \end{array}
\right).
\end{equation}
Then the sequence $y$ verifies
\begin{equation}
\label{recOrdSup4}
y(t+{\bf 1})
=
A(t)y(t)+b(t),
\quad \forall t \in \mathbb{N}^m,
\end{equation}
being given
$y(t^1, \ldots, t^{\beta-1},0,t^{\beta+1}, \ldots, t^{m})$, $\forall
\beta\in \{1,2, \ldots, m\}$.

\section{Second order multivariate recurrences}

Apart from the first order multitime recurrences, an important role is played by the second order multitime recurrences.

\begin{lemma}
\label{rec.lem2}
Let $A\in \mathcal{M}_{2}(\mathbb{R})$ and
$\lambda_1, \lambda_2$ its eigenvalues, i.e., the roots of the polynomial
\begin{equation}
\label{ecl2.1}
P(\lambda)=\lambda^2 - \mathrm{Tr} (A) \lambda +\det(A).
\end{equation}

\noindent
$i)$ If $\lambda_1, \lambda_2 \in \mathbb{R}$
and $\lambda_1 \neq \lambda_2$, then $\forall k \in \mathbb{N}$
\begin{equation}
\label{ecl2.2}
A^k=\frac{\lambda_1^k-\lambda_2^k}{\lambda_1-\lambda_2}\,A-
\frac{\lambda_2\lambda_1^k-\lambda_1\lambda_2^k}{\lambda_1-\lambda_2}\,I_2.
\end{equation}

\noindent
$ii)$ If $\lambda_1 = \lambda_2 $, then $\forall k \in \mathbb{N}$
\begin{equation}
\label{ecl2.3}
A^k=k\lambda_1^{k-1}\,A-(k-1)\lambda_1^k\,I_2.
\end{equation}

\noindent
$iii)$ If $\lambda_1, \lambda_2 \in \mathbb{C} \setminus \mathbb{R}$,
$\lambda_1, \lambda_2=r(\cos \theta \pm i\sin \theta)$,
with $r>0$, $\theta \in (0,2\pi) \setminus \{\pi\}$,
then $\forall k \in \mathbb{N}$
\begin{equation}
\label{ecl2.4}
A^k=\frac{r^{k-1}\sin k\theta}{\sin \theta}\,A-
\frac{r^{k}\sin (k-1)\theta}{\sin \theta}\,I_2.
\end{equation}
\end{lemma}
\begin{proof}
There exists $Q(X)\in \mathbb{R}[X]$ and
$c_0, c_1 \in \mathbb{R}$ such that
\begin{equation*}
X^k=P(X)Q(X)+c_1X+c_0.
\end{equation*}
Hence $A^k=P(A)Q(A)+c_1A+c_0I_2$. Since $P(A)=O_2$, we obtain
\begin{equation*}
A^k=c_1A+c_0I_2.
\end{equation*}
If $\lambda_1 \neq \lambda_2$,
for $X=\lambda_1$, $X=\lambda_2$ ($P(\lambda_j)=0$), then
\begin{equation*}
\lambda_1^k=c_1\lambda_1+c_0,
\quad
\lambda_2^k=c_1\lambda_2+c_0,
\end{equation*}
whence we find
\begin{equation*}
c_1=\frac{\lambda_1^k-\lambda_2^k}{\lambda_1-\lambda_2},
\quad
c_0=-
\frac{\lambda_2\lambda_1^k-\lambda_1\lambda_2^k}{\lambda_1-\lambda_2}.
\end{equation*}
Consequently, we get the results in $i)$ and $iii)$.

If $\lambda_1 = \lambda_2$, then $\lambda_1$ is a double root of the polynomial
$S(X):=X^k-c_1X-c_0$. Hence $S(\lambda_1)=0$ and $S'(\lambda_1)=0$. From these relations it follows
$c_1=k\lambda_1^{k-1}$, $c_0=-(k-1)\lambda_1^k$ and one obtains the result in $ii)$.
\end{proof}

\vspace{0.2 cm}
We consider now diagonal recurrences of order two
\begin{equation}
\label{recOrd2}
x(t+2 \cdot {\bf 1})+ax(t+{\bf 1})+bx(t)=0,
\end{equation}
with $a,b \in \mathbb{R}$ constants and with initial conditions
$$x(t^1, \ldots, t^{\beta-1},0,t^{\beta+1}, \ldots, t^{m}),\,\, x(t^1, \ldots, t^{\beta-1},1,t^{\beta+1}, \ldots, t^{m})$$
given for any
$\beta \in \{1,2,\ldots, m\}$
and
$x\colon \mathbb{N}^m \to \mathbb{R}$
unknown sequence. Denoting
\begin{equation*}
y(t)=
\left(
  \begin{array}{c}
    x(t) \\
    x(t+{\bf 1}) \\
  \end{array}
\right),
\quad
A=
\left(
  \begin{array}{cc}
    0 & 1 \\
    -b & -a \\
  \end{array}
\right)
\end{equation*}
the multivariate sequence $y$ verifies
$y(t+{\bf 1})= Ay(t)$. From the Corollary of the Theorem \ref{rec.t1}, it follows
\begin{equation*}
\begin{split}
&y(t)=
\left(
  \begin{array}{c}
    x(t) \\
    x(t+{\bf 1}) \\
  \end{array}
\right)
=
A^{t^\beta}y(t-t^\beta \cdot {\bf 1})
=
A^{t^\beta}\left(
  \begin{array}{c}
    x(t-t^\beta \cdot {\bf 1}) \\
    x(t-(t^\beta-1) \cdot {\bf 1}) \\
  \end{array}
\right),\\
&\mbox{ if }\,\,
\mu(t)=t^\beta.
\end{split}
\end{equation*}

The characteristic polynomial
$P(\lambda)=\lambda^2+a\lambda +b$
of the matrix $A$ has the roots $\lambda_1, \lambda_2$.
According the Lemma \ref{rec.lem2}, the matrix $A^{t^\beta}$ is of the form
$A^{t^\beta}=c_1(t^\beta)A+c_0(t^\beta)I_2$ (formulas
(\ref{ecl2.2}), (\ref{ecl2.3}), (\ref{ecl2.4})). Hence
\begin{equation*}
\begin{split}
\left(
  \begin{array}{c}
    x(t) \\
    x(t+{\bf 1}) \\
  \end{array}
\right)
&=c_1(t^\beta)\,A  \left(
  \begin{array}{c}
    x(t-t^\beta \cdot {\bf 1}) \\
    x(t-(t^\beta-1) \cdot {\bf 1}) \\
  \end{array}
\right)
+
c_0(t^\beta)\left(
  \begin{array}{c}
    x(t-t^\beta \cdot {\bf 1}) \\
    x(t-(t^\beta-1) \cdot {\bf 1}) \\
  \end{array}
\right),\\
\quad\mbox{ if }\,\,
\mu(t)&=t^\beta.
\end{split}
\end{equation*}
Consequently
\begin{equation*}
x(t)
=c_1(t^\beta)x(t-(t^\beta-1) \cdot {\bf 1})
+
c_0(t^\beta)x(t-t^\beta \cdot {\bf 1}),\,\,
\mbox{ if }\,\,
\mu(t)=t^\beta.
\end{equation*}

We have proved the following result
\begin{theorem}
\label{rec.t2}
Let $m\geq 2$, $a,b \in \mathbb{R}$ and $\lambda_1, \lambda_2$
the roots of the polynomial $P(\lambda)=\lambda^2+a\lambda +b$.
Suppose that the $(m-1)$-sequences
$$f_1, f_2, \ldots, f_m \colon \mathbb{N}^{m-1} \to \mathbb{R},\,\,
g_1, g_2, \ldots, g_m \colon \mathbb{N}^{m-1} \to \mathbb{R},$$
satisfy, for any $\alpha,\beta \in \{1, 2, \ldots, m\}$, the compatibility conditions
\begin{equation}
\label{ect2.1}
\begin{split}
f_{\alpha}(t^1,\ldots, \widehat{t^\alpha}, \ldots, t^m)\Big|_{t^\beta=0}
=
f_{\beta}(t^1,\ldots, \widehat{t^\beta}, \ldots, t^m)\Big|_{t^\alpha=0},\\
g_{\alpha}(t^1,\ldots, \widehat{t^\alpha}, \ldots, t^m)\Big|_{t^\beta=1}
=
g_{\beta}(t^1,\ldots, \widehat{t^\beta}, \ldots, t^m)\Big|_{t^\alpha=1},\\
f_{\alpha}(t^1,\ldots, \widehat{t^\alpha}, \ldots, t^m)\Big|_{t^\beta=1}
=
g_{\beta}(t^1,\ldots, \widehat{t^\beta}, \ldots, t^m)\Big|_{t^\alpha=0},\\
\forall
t^1, \ldots, t^{\alpha-1}, t^{\alpha+1}, \ldots, t^{\beta-1}, t^{\beta+1}, \ldots, t^m
\in \mathbb{N}.
\end{split}
\end{equation}
Then the unique $m$-sequence
$x \colon \mathbb{N}^m \to \mathbb{R}$ which verifies
\begin{equation}
\label{ect2.2}
\begin{split}
&x(t+2 \cdot {\bf 1})+ax(t+{\bf 1})+bx(t)=0,
\quad \forall t \in \mathbb{N}^m,\\
&x(t)\Big|_{t^\gamma=0}
=
f_{\gamma}(t^1,\ldots, \widehat{t^\gamma}, \ldots, t^m),
\quad
\forall
(t^1,\ldots, \widehat{t^\gamma}, \ldots, t^m)\in \mathbb{N}^{m-1},\\
&x(t)\Big|_{t^\gamma=1}
=
g_{\gamma}(t^1,\ldots, \widehat{t^\gamma}, \ldots, t^m),
\quad
\forall
(t^1,\ldots, \widehat{t^\gamma}, \ldots, t^m)\in \mathbb{N}^{m-1},\\
& \qquad \forall \gamma\in \{1,2, \ldots, m\},
\end{split}
\end{equation}
is defined by the following formulas:

\noindent
$i)$ If $\lambda_1, \lambda_2 \in \mathbb{R}$
and $\lambda_1 \neq \lambda_2$, then
\begin{equation}
\label{ect2.3}
\begin{split}
x(t)
&=\frac{\lambda_1^{t^\beta}-\lambda_2^{t^\beta}}{\lambda_1-\lambda_2}\,
g_{\beta}(t^1-t^\beta+1, \ldots, \widehat{t^\beta}, \ldots, t^{m-1}-t^\beta+1)\\
&-\frac{\lambda_2\lambda_1^{t^\beta}-\lambda_1\lambda_2^{t^\beta}}{\lambda_1-\lambda_2}\,
f_{\beta}(t^1-t^\beta, \ldots, \widehat{t^\beta}, \ldots, t^{m-1}-t^\beta),\,
\mbox{ if }\,\, \mu(t)=t^\beta.
\end{split}
\end{equation}

\noindent
$ii)$ If $\lambda_1 = \lambda_2 $, then
\begin{equation}
\label{ect2.4}
\begin{split}
x(t)
&=t^\beta \lambda_1^{t^\beta-1}\,
g_{\beta}(t^1-t^\beta+1, \ldots, \widehat{t^\beta}, \ldots, t^{m-1}-t^\beta+1)\\
&-(t^\beta-1)\lambda_1^{t^\beta}\,
f_{\beta}(t^1-t^\beta, \ldots, \widehat{t^\beta}, \ldots, t^{m-1}-t^\beta),\,
\mbox{ if }\,\, \mu(t)=t^\beta.
\end{split}
\end{equation}

\noindent
$iii)$ If $\lambda_1, \lambda_2 \in \mathbb{C} \setminus \mathbb{R}$,
$\lambda_1, \lambda_2=r(\cos \theta \pm i\sin \theta)$,
with $r>0$, $\theta \in (0,2\pi) \setminus \{\pi\}$, then
\begin{equation}
\label{ect2.5}
\begin{split}
x(t)
&=\frac{r^{t^\beta-1}\sin t^\beta \theta}{\sin \theta}\,
g_{\beta}(t^1-t^\beta+1, \ldots, \widehat{t^\beta}, \ldots, t^{m-1}-t^\beta+1)\\
&-\frac{r^{t^\beta}\sin (t^\beta-1)\theta}{\sin \theta}\,
f_{\beta}(t^1-t^\beta, \ldots, \widehat{t^\beta}, \ldots, t^{m-1}-t^\beta),\,
\mbox{ if }\,\, \mu(t)=t^\beta.
\end{split}
\end{equation}
\end{theorem}

\section{Discrete minimal submanifolds}

Let $(M,g)$ be a Riemannian manifold and $(N,h)$ be a Riemannian oriented submanifold (possibly with boundary).
Let $x=(x^i),\,i=1,...,n,$ be the local coordinates in $M$ and $t=(t^\alpha),\,\alpha=1,...,m,$
the local coordinates in $N$. If the parametric equations of the submanifold $N$ are
$x^i = x^i(t)$, then the induced Riemannian metric has the components
$$h_{\alpha\beta}(t)= g_{ij}(x(t))x^i_\alpha(t)x^j_\beta(t),$$
where $x^i_\alpha(t)=\frac{\partial x^i}{\partial t^\alpha}(t)$.
The determinant of this metric is denoted by $d = \det(h_{\alpha\beta}(t))$.
If $\Sigma \subset N$ is compact subset, corresponding to $t\in \Omega\hbox{- compact}$, then its area is
$$\int_\Sigma d\sigma= \int_\Omega \sqrt{d}\,\,dt^1\wedge...\wedge dt^m.$$
The submanifold $\Sigma$ is called {\it minimal} if and only if it is a critical
point of the area functional
$$I(x(\cdot))=\int_\Omega \sqrt{d}\,\,dt^1\wedge...\wedge dt^m,$$
for all compactly supported variations. Introducing the Lagrangian
$$L= \sqrt{d}= \sqrt{\det\left(g_{ij}(x(t))x^i_\alpha(t)x^j_\beta(t)\right)}$$
a minimal submanifold is solution of Euler-Lagrange PDEs system
$$\frac{\partial L}{\partial x^i}- D_\alpha\frac{\partial L}{\partial x^i_\alpha}=0,$$
i.e., vanishing {\it mean curvature vector},
$$\sum_r\Omega_{r|\alpha\beta}h^{\alpha\beta}\xi_r^i=0.$$
Though continuous models of minimal submanifolds are usually more convenient and yield results which are
more transparent, the discrete models are also of interest being in fact discrete dynamical systems.

The theory of integrators for multi-parameter Lagrangian dynamics shows that
instead of discretization of Euler-Lagrange PDEs we must use a discrete Lagrangian,
a discrete action, and then discrete Euler-Lagrange equations (see \cite{MW}). Of course, the discrete
Euler-Lagrange equations associated to multitime discrete Lagrangian can be solved
successfully by the Newton method if it is convergent for a convenient step.

To simplify, we consider the minimal $2$-dimensional submanifols, having the coordinates
$t=(t^1,t^2)$. The discretization of the Lagrangian $L(x(t),x_\alpha(t))$ can be performed
by using the centroid rule (see \cite{U1}, \cite{U7}, \cite{UAC}) which consists in: (i) the substitution
of the point $(t^1,t^2)$ with $(mh^1,nh^2)$, for the fixed step $(h^1, h^2)$; (ii) the substitution
of the point $x(t^1,t^2)$ with the fraction
$$\xi_{mn}=\frac{x_{m\,n}+x_{m+1\,n} +x_{m\, n +1}}{3};$$
and (iii) the substitution of the partial velocities $x_1=\frac{\partial x}{\partial t^1}$,
$x_2=\frac{\partial x}{\partial t^2}$ by the fractions
$\frac{x_{m+1\,n}-x_{m\,n}}{h_1}$, $\frac{x_{m\,n+1}-x_{m\,n}}{h_2}.$
We can write
$$L^2_d= \det{h_{\alpha\beta}(m,n)}$$
and since
$$h_{\alpha \beta} = g_{ij}( \xi _{mn}) \frac {(x^i_{mn})_\alpha }{h^\alpha} \frac {(x^j_{mn})_\beta }{h^\beta},\, (x^i_{mn})_1=x^i_{m+1\,n}-x^i_{mn},\,
(x^i_{mn})_2=x^i_{m\,n+1}-x^i_{mn},$$
it follows
$$h_1^2h_2^2\,L^2_d = \det \left(g_{ij}(\xi_{mn})(x^i_{m\,n})_\alpha (x^j_{mn})_\beta\right).$$
The discrete Euler-Lagrange equations are
$$\sum_\xi \frac{\partial L_d}{\partial x_{mn}}(\xi)=0$$
where $\xi$ runs over three points,
$$(x_{m\,n}, x_{m+1\,n},x_{m\,n+1})$$
$$(x_{m-1\,n}, x_{m\,n},x_{m-1\,n+1}),\, \hbox{left shift map}$$
$$(x_{m\,n-1}, x_{m+1\,n-1},x_{m\,n}),\, \hbox{right shift map}$$
and
$$m=1,...,M-1,\,\,n=1,...,N-1.$$

Since
$$\frac{\partial(x^i_{mn})_\alpha}{\partial x^k_{mn}}=-\delta^i_k,\,\,\frac{\partial(x^i_{m-1\,n})_\alpha}{\partial x^k_{mn}}=\delta ^i_k \delta ^1_\alpha ,\,\,\frac{\partial(x^i_{m\,n-1})_\alpha}{\partial x^k_{mn}}=\delta ^i_k \delta ^2_\alpha ,$$
explicitly, we compute
$$2\,\frac{\partial L_d}{\partial x^k_{mn}}(x_{m\,n}, x_{m+1\,n},x_{m\,n+1})
= \frac{1}{\sqrt{d}}\,\,\frac{\partial d}{\partial h_{\alpha\beta}}\,\,\frac{\partial h_{\alpha\beta}}{\partial x^k_{mn}}$$
$$= \sqrt{d}\,h^{\alpha\beta}\left(\frac{1}{3}\frac{\partial g_{ij}}{\partial x^k}(\xi_{mn})(x^i_{m\,n})_\alpha (x^j_{m\,n})_\beta
-g_{kj}\left((x^j_{m\,n})_\alpha +(x^j_{m\,n})_\beta\right)\right);$$

$$2\,\frac{\partial L_d}{\partial x^k_{mn}}(x_{m-1\,n}, x_{m\,n},x_{m-1\,n+1})$$
{\small $$= \sqrt{d}\,h^{\alpha\beta}\left(\frac{1}{3}\frac{\partial g_{ij}}{\partial x^k}(\xi_{m-1\,n})(x^i_{m-1\,n})_\alpha(x^j_{m-1\,n})_\beta
+g_{kj}\left((x^j_{m-1\,n})_\alpha \delta^1_\beta + \delta^1_\alpha (x^j_{m-1\,n})_\beta \right)\right);$$}

$$2\,\frac{\partial L_d}{\partial x^k_{mn}}(x_{m\,n-1}, x_{m+1\,n-1},x_{m\,n})$$
{\small $$= \sqrt{d}\,h^{\alpha\beta}\left(\frac{1}{3}\frac{\partial g_{ij}}{\partial x^k}(\xi_{m\,n-1})(x^i_{m\,n-1})_\alpha (x^j_{m\,n-1})_\beta
+g_{kj}\left((x^j_{m\,n-1})_\alpha \delta^2_\beta + \delta^2_\alpha (x^j_{m\,n-1})_\beta\right)\right).$$}

\begin{theorem}
Denoting
$$A_{mn\alpha\beta}= - g_{kj}\left((x^j_{m\,n})_\alpha +(x^j_{m\,n})_\beta\right)$$
$$A_{m-1\,n\alpha\beta}= g_{kj}\left((x^j_{m\,n})_\alpha \,\delta^1_\beta +\delta^1_\alpha\, (x^j_{m\,n})_\beta\right)$$
$$A_{m\,n-1\alpha\beta}= g_{kj}\left((x^j_{m\,n})_\alpha\, \delta^2_\beta + \delta^2_\alpha \,(x^j_{m\,n})_\beta \right),$$
the variational integrator of discrete minimal $2$-submanifolds is described by the recurrence equation
$$
\sum_{m,m-1,n-1}\,h^{\alpha\beta}(m,n)\left(\frac{1}{3}\frac{\partial g_{ij}}{\partial x^k}(\xi_{mn})(x^i_{m\,n})_\alpha (x^j_{m\,n})_\beta + A_{mn\alpha\beta}\right)=0.$$
\end{theorem}

\section{Conclusions}
This paper presents original results regarding the multivariate recurrence equations.
Our approach to multivariate recurrence equations is advantageous for practical problems.
The original results have a great potential to solve problems in various areas such as
ecosystem dynamics, financial modeling, economics, image processing (representations of filters),
and differential geometry etc. The two-dimensional filters are extensively used
in processing two-dimensional sampled data (seismic data sections, digitized photographic
data, gravitational and magnetic maps etc).

\vspace{0.3cm}

{\bf Acknowledgements}

Partially, the work has been funded by the Sectoral Operational Programme Human Resources
Development 2007-2013 of the Ministry of European Funds through
the Financial Agreement POSDRU/159/1.5/S/132395.

Partially supported by University Politehnica of Bucharest and by Academy of Romanian Scientists.

Parts of this paper were presented at X-th International Conference
on Finsler Extensions of Relativity Theory (FERT 2014) August 18-24, 2014, Bra\c sov, Romania
and at The VIII-th International Conference ``Differential Geometry and Dynamical Systems"
(DGDS-2014), 1 - 4 September 2014, Mangalia, Romania.



\begin{thebibliography}{99}

\bibitem{Ah} M. Ahmadi,
{\it Design of 2-Dimensional recursive digital filters},
Control and Dynamics System, 78 (1996), 131-181.

\bibitem{BP} M. Bousquet-M\' elou, M. Petkov\v sek,
{\it Linear recurrences with constant coefficients: the multivariate case},
Discrete Mathematics 225, 1 (2000), 51-75.

\bibitem{CP} M. C\^{\i}rnu, V. Prepeli\c t\u a,
{\it Hybrid Laplace Transformation},
International Journal of Scientific and Innovative Mathematical Research (IJSIMR)
2, 2 (2014), 209-220.

\bibitem{C} P. Concus,
{\it Numerical solution of the minimal surface equation},
Math. Comp. 21 (1967), 340-350.

\bibitem{DM} D. Dudgeon, R. Mersereau,
{\it Multidimensional Digital Signal Processing},
Prentice-Hall Signal Processing Series, 1983.

\bibitem{E} S. Elaydi,
{\it An Introduction to Difference Equations}, Springer, 2005.


\bibitem{FM} E. Fornasini, G. Marchesini,
{\it State-space realization theory of two-dimensional filters},
IEEE Trans. Automat. Control AC-21 (4) (1976) 484-491.

\bibitem{FoM} E. Fornasini, G. Marchesini,
{\it Doubly-indexed dynamical systems: state-space models and structural models},
Math. Systems Theory, 12 (1978) 59-72.

\bibitem{FMa} E. Fornasini, G. Marchesini,
{\it Stability analysis of 2-D systems},
IEEE Trans. Circuits and Systems 27, 12 (1980), 1210-1217.


\bibitem{HK} H. Hauser, C. Koutschan,
{\it Multivariate linear recurrences and power series division},
Discrete Mathematics, 312 (2012), 3553-3560.

\bibitem{Hu} T. S. Huang,
{\it Stability of two-dimensional recursive filters},
IEEE Transactions on Audio and Electroacoustics, 20, 2 (1972), 158-163.


\bibitem{JL} J. S. Lim,
{\it Two-Dimensional Signal and Image Processing},
Prentice-Hall Inc., Englewood Cliffs, NJ, 1990.

\bibitem{MW} J. E. Marsden, M. West,
{\it Discrete mechanics and variational integrators},
Acta Numerica 10 (2001), 1-158.

\bibitem{MGS} N. Mastorakis, I. Gonos, M. Swamy,
{\it Design of two-dimensional recursive filters using genetic algorithms},
IEEE Transactions on Circuits and Systems I Fundamental Theory and Applications, 50, 5 (2003), 634-639.

\bibitem{MA} A. Mazinani, M. Ahmadi, M. Shridhar and R. S. Lashkari,
{\it A novel approach to the design of 2-D recursive digital filters},
Journal of the Franklin Institute, Pergamon Press Ltd, 329, 1 (1992), 127-133.

\bibitem{PP} U. Pinkall, K. Polthier,
{\it Computing discrete minimal surfaces and their conjugates},
Experiment. Math. 2, 1 (1993), 15-36.

\bibitem{Pr} V. Prepeli\c t\u a,
{\em  Multiple $(n,m)$-hybrid Laplace Transformation
and applications to multidimensional hybrid systems. Part I},
U.P.B. Sci. Bull., Series A, 72, 2 (2010), 105-120.

\bibitem{Ro} R. Roesser,
{\it A discrete state -space model for linear image processing},
IEEE Trans. Automat. Control AC-20, 1 (1975), 1-10.

\bibitem{U1} C. Udri\c ste, M. Postolache, I. \c Tevy,
{\it Integrator for Lagrangian dynamics},
Balkan J. Geom. Appl. 6, 2 (2001), 109-115.

\bibitem{U2} C. Udri\c ste,
{\it Multitime maximum principle for curvilinear integral cost},
Balkan J. Geom. Appl., 16, 1 (2011), 128-149.

\bibitem{U3} C. Udri\c ste, A. Bejenaru,
{\it Multitime optimal control with area integral costs on boundary},
Balkan J. Geom. Appl., 16, 2 (2011), 138-154.


\bibitem{U4} C. Udri\c ste, I. \c Tevy,
{\it Multitime dynamic programming for multiple integral actions},
Journal of Global Optimization, 51, 2 (2011), 345-360.

\bibitem{U4b} C. Udri\c ste, V. Damian, L. Matei, I. \c Tevy,
{\it Multitime differentiable stochastic processes, diffusion PDEs, Tzitzeica hypersurfaces},
U.P.B. Sci. Bull., A, 74, 1 (2012), 3-10.

\bibitem{U5} C. Udri\c ste, {\it Minimal submanifolds and harmonic maps through
multitime maximum principle}, Balkan J. Geom. Appl., 18, 2 (2013), 69-82.

\bibitem{U6} C. Udri\c ste, S. Dinu, I. \c Tevy,
{\it Multitime optimal control for linear PDEs with curvilinear cost functional},
Balkan J. Geom. Appl., 18, 1 (2013), 87-100.

\bibitem{UAC} C. Udri\c ste, V. Arsinte, C. Cipu,
{\it Von Neumann analysis of linearized discrete Tzitzeica PDE},
Balkan J. Geom. Appl., 15, 2 (2010), 100-112.

\bibitem{U7} C. Udri\c{s}te, M. Ferrara, D. Opris,
{\it Economic Geometric Dynamics},
Monographs and Textbooks 6, Geometry Balkan Press, Bucharest, 2004.

\bibitem{YN} B. Yuttanan, C. Nilrat,
{\it Roots of matrices}, Songklanakarin J. Sci. Technol., 27, 3 (2005), 659-665.

\end{thebibliography}
\end{document}